\documentclass[twoside,reqno,10pt]{amsart}

 \usepackage{amsmath}
	
	\newcommand{\diff}[4]{\ensuremath{ \mathrm{d}^{#4 }_{ #3 } \left[ #1 \right] \left(  #2 \right)  }}
	\newcommand{\diffplus}[3]{\diff{#1}{#2}{#3}{+} }
	\newcommand{\diffmin}[3]{\diff{#1}{#2}{#3}{-} }

	\newcommand{\fdiff}[4]{\ensuremath{ \upsilon^{#4 }_{ #3 }  #1  \left(  #2 \right)  }}
	\newcommand{\fdiffplus}[3]{ \fdiff {#1}{#2}{+}{#3} }
	\newcommand{\fdiffmin}[3]{ \fdiff {#1}{#2}{-}{#3} }
	\newcommand{\fdiffpm}[3]{ \fdiff {#1}{#2}{\pm}{#3} }
	 \newcommand{\fracvar}[4]{\ensuremath{ \upsilon_{ #3 }^{ #4} \left[ #1 \right] \left(  #2 \right)   }}
	 \newcommand{\fracvarplus}[3]{ \fracvar {#1}{#2}{#3}{\epsilon+} }
	 \newcommand{\fracvarmin}[3]{ \fracvar {#1}{#2}{#3}{\epsilon -} }
	 
	\newcommand{\llim}[3]{\ensuremath{ \lim\limits_{ #1 \rightarrow #2} #3 }}
	\newcommand{\fclass}[2]{\ensuremath{  \mathbb{#1}^{\, #2} }}
	
		\newcommand{\holder}[1]{\fclass{H}{#1} }

	\newcommand{\deltaop}[4]{\ensuremath{ \Delta_{ #3 }^{ #4} \left[ #1 \right] \left(  #2 \right)   }}
	\newcommand{\deltaplus}[2]{ \deltaop {#1}{#2}{\epsilon}{+} }
	\newcommand{\deltamin}[2]{ \deltaop {#1}{#2}{\epsilon}{-} }
		 
	\newcommand{\sep}{,}
	\newcommand{\epnt}{\; .}
	\newcommand{\ecma}{\; ,}

\renewcommand{\dj} {\leavevmode\raisebox{-.05ex}{\makebox[0pt][l]{\hskip-.05em\accent"16\hss}}d \,}

 \newtheorem{theorem}{Theorem}
 \newtheorem{lemma}{Lemma}
 \newtheorem{corollary}{Corollary}
 \newtheorem{proposition}{Proposition}

\newtheorem{definition}{Definition}
 \newtheorem{remark}{Remark}
 \newtheorem{example}{Example}

 \usepackage{graphicx,color,psfrag, amssymb}
 
 \usepackage{lineno}

\begin{document}
  
   \title[Taylor expansions of H\"olderian functions]{Regularized and Fractional Taylor expansions of H\"olderian functions\thanks{The work has been supported in part by a grant from Research Fund - Flanders (FWO), contract number 0880.212.840.}}
   
   \author {Dimiter Prodanov}
   \address{Correspondence: Environment, Health and Safety, IMEC vzw, Kapeldreef 75, 3001 Leuven, Belgium Correspondence: Environment, Health and Safety, IMEC vzw, Kapeldreef 75, 3001 Leuven, Belgium 
   e-mail: Dimiter.Prodanov@imec.be, dimiterpp@gmail.com   }


   	\begin{abstract}
   	 H\"olderian functions have strong non-linearities, which result in singularities in the derivatives.  
   	 This manuscript presents several fractional-order Taylor expansions of H\"olderian functions around points of non- differentiability.
   	 These expansions are derived using the concept of a fractional velocity, 
   	 which can be used to describe the singular behavior of derivatives and non-differentiable functions.  
   	 Fractional velocity is defined as the limit of the difference quotient of the increment of a function and the difference of its argument raised to a fractional power.
   	 Fractional velocity can be used to regularize ordinary derivatives. 
   	 To this end, it is possible to define regularized Taylor series and compound differential rules. 
   	 In particular a compound differential rule for Holder 1/2 functions is demonstrated.
	 The expansion is presented using the auxiliary concept of fractional co-variation of functions. 
	 
   	 \medskip
   		
   		
 \smallskip
   		

 \end{abstract}

	   	\maketitle  
	   	
   	
   	\section{Introduction}
   	\label{seq:intro}
   	
   	Difference quotients of functions of fractional order have been considered initially by du Bois-Reymond  \cite{BoisReymond1875} and later by Faber \cite{Faber1909} in their studies of the point-wise differentiability of functions.
   	The concept implies what is now known as Holder-continuity of the function.
   	Later  Cherbit \cite{Cherbit1991} introduced the notion of \textsl{fractional velocity} as the limit of the fractional difference quotient. His main application was the study of fractal phenomena and physical processes for which the instantaneous velocity was not well defined. 
   	Independently,  applications to physical systems exhibiting fractal behavior have  inspired the development of local integral definitions of fractional derivatives.  
   	The definition of \textsl{local fractional derivative} introduced by Kolwankar and Gangal \cite{Kolwankar1997a} is based on the localization of Riemann-Liouville fractional derivatives towards a particular point of interest. Manipulations of this construct in many occasions  are complicated and there are only few physical applications so far. 
	On the other hand,  Ben Adda and Cresson  \cite{Adda2001},\cite{Adda2013} demonstrated equivalence of so-defined local fractional derivative with an alternative definition  based on a limit of a certain differential quotient. 
	The correspondences between the integral and the quotient difference approaches have been further investigated by  Chen et al. \cite{Chen2010} and some of the results stated initially by Ben Adda and Cresson have been clarified recently \cite{Adda2013}.    
    
 	The present work builds on the \textsl{fractional variation} operators, which in limit are
 	equivalent to the \textsl{fractional velocity} introduced by Cherbit \cite{Cherbit1991}.
 	Previous work has demonstrated some basic algebraic properties of these operators and established their limiting behavior \cite{Prodanov2015}. 

   	The main application is a formal regularization procedure for a derivative of a H\"olderian function, which allows for removal of weak singularities in the derivative caused by strong non-linearities. 
   	This becomes possible by exploiting some of the properties of the  \textsl{fractional variation} and velocity. 
   	One of main results of this procedure is the Taylor expansion of a H\"olderian function, termed the 
   	\textsl{Taylor -- It\^o} expansion, which makes use of a regularization of the integer-order derivative.
   	This can be stated as follows:
   	For a compound function $f(w, x)$ where $w(x)$ is \holder{1/2} the following expansion holds:
 \[
	\frac{\dj^{ \pm} }{dx} f(x,w) =\frac{\partial f}{\partial x} + \frac{\partial f}{\partial w}   \;  \frac{\dj^{\pm} }{dx } w (x)    \pm \frac{1}{2}   \frac{\partial ^2 f}{\partial w^2} [ w, w]^{\pm} (x) \epnt   
\] 
    		
   	The approach presented here can be used to characterize and approximate functions about points for which  integer-order derivatives are not defined. 
   	Such functions are considered in the physical theory of scale relativity \cite{Nottale1989},\cite{Nottale1998}, \cite{Nottale2011} where geodesic trajectories are considered to be of order \holder{1/2}.

   	The manuscript is organized as follows:
   	Section \ref{sec:definitions} gives general definitions and notational conventions.
   	Section \ref{sec:frdiff} introduces fractional (fractal) variation and fractional velocity.
   	Section \ref{sec:taylorexp} demonstrates fractional and mixed-order Taylor expansions.
    Section \ref{sec:reg} presents ordinary derivative regularization
   and demonstrates the compound differential rule for \holder{1/2} functions.


\section{Preliminaries}
  	\label{sec:definitions}
  	  	
  	The term \textit{function}  denotes a mapping $ f: \mathbb{R} \mapsto \mathbb{R} $ or in some cases $\mathbb{C} \mapsto \mathbb{C}$. 
  	The notation $f(x)$ is used to refer to the value of the function at the point \textit{x}.
    The term \textit{operator}  denotes the mapping from functional expressions to functional expressions.
  	The symbol \fclass{C}{0} denotes the class of continuous functions. 
  	The symbol \fclass{C}{n} -- the class of \textit{n}-times differentiable functions where  $n  \in \mathbb{N}$.
  	Square brackets are used for the arguments of operators, while round brackets are used for the arguments of functions.	
    $Dom[f]$ denotes the domain of definition of the function $f(x)$.
 
   	\begin{definition}
   		\label{def:holder}
   		Let \holder{\alpha} be the class of H\"older  continuous  functions of degree $\alpha$, $\alpha \in (0,\, 1]$.
   		That is, we say that $f$ is of class \holder{\alpha} if $\; \forall f \in \holder{\alpha} $ there exist two non-negative constants 
   		$C, \delta \in \mathbb{R} $, such that  for two  values $  x, y \in Dom[ f ]$,  for which $|x-y| \leq \delta, \  x \neq y$, the following inequality holds
   		\[
   		| f (x) - f (y) |  \leq C |x-y|^\alpha \epnt
   		\]  		
   		For mixed orders $n +\alpha>1 $,  $n \in \fclass{N}{}$,  the following extended definition holds: 
   		 H\"older class \holder{n+ \alpha}  designates the class of \fclass{\, C}{n} functions  for which the inequality 
   		\[
   		| f (x) - f (y) - P_n (x-y) |  \leq C |x-y|^{n +\alpha} \ecma
   		\]
   		holds in the interval $ \left[ x, y\right] $. $P_n (.)  $ designates a real-valued polynomial of degree $n$ of the form
   		$
   		P_n (z) = \sum\limits_{k=1}^{n}{ a_k z^k} 
   		$,
   		where $P_0(z) = 0$ and $\alpha \in (0,\, 1)$.
   	\end{definition}
	 \begin{remark}
   	Throughout the rest of the paper it will be assumed that, unless stated otherwise, the  H\"older exponent characterizing the function is point-wise unique. 
   	That is, there is only \textbf{one} fractional exponent associated with a point belonging to the domain of the function. 
   	
   	In the opposite case we will discuss functions characterized by a \textbf{H\"older spectrum}, that is the set of such exponents. 
   	Functions having more than one exponent will be denoted to as multi-H\"olderian and if necessary designated by the notation  \holder{  \alpha_1 + \alpha_2 + \ldots}. 
   	Indeed, it can be shown that if
   	$
   	 f(x) \in \holder{\alpha}
   	$   
   	then 
   	$
   	f(x) \in \holder{\alpha+ \beta}
   	$  if $ \beta > \alpha$  
   	implying that $\holder{\alpha} \subseteq \holder{\alpha +\beta}$.
   	
   	Therefore, unless stated explicitly the assignment $f(x) \in \holder{\alpha}$
   	  will \textbf{not denote} sums of functions. 
    \end{remark}
  
  	\begin{definition}
  		\label{def:deltas}
  		Let the parametrized difference operators acting on a function $f(x)$ be defined in the following way
  		\begin{flalign*}
	  	\deltaplus{f}{x}   & :=  f(x + \epsilon) - f(x) \ecma\\
  		\deltamin{f}{x} & :=  f(x) - f(x - \epsilon)  \ecma \\
  		\deltaop {f}{x}{\epsilon}{2}  &:=  f(x + \epsilon) -2 f(x) + f(x - \epsilon) \ecma
  		\end{flalign*}
  		where $\epsilon>0$. The first one we refer to as \textit{forward difference} operator, 
  		the second one we refer to as \textit{backward difference} operator and the third one as \textit{2\textsuperscript{nd} order difference} operator.
  	\end{definition}
  	
   	\section{Fractional (fractal) variation and fractional velocity of functions}
   	\label{sec:frdiff}
    Fractional (fractal) variation operators have been introduced in a previous work \cite{Prodanov2015} as: 
   	\begin{definition}
   		\label{def:fracvar}
   		Let the \textit{Fractal Variation} operators of order $\beta$ be defined as
   		\begin{align}
   			\label{eq:fracvar1}
   			\fracvarplus {f}{x}{\beta} := \frac{ \deltaplus{f}{x}  }{\epsilon ^\beta}  =\frac{ f(x+ \epsilon) - f(x) }{\epsilon ^\beta} \ecma
   			\\
   			\fracvarmin {f}{x}{\beta} :=  \frac{ \deltamin{f}{x} }{\epsilon ^\beta}  =\frac{ f(x)- f( x- \epsilon)  }{\epsilon  ^\beta} \ecma
   		\end{align}
   		where  $\epsilon >0$ and $0 < \beta \leq 1 $ are real parameters and $f(x)$ is a function.
   	\end{definition}
   In this paper we will consider the following extension of the previous definition
 
   \begin{definition}
   	\label{def:fracvarn}
   	Let the \textsl{fractional variation} operators of  mixed order $n+\beta$ acting on a function $f(x)$ be  defined as
   	\begin{align}
   		\label{eq:fracdiffn}
   		\fracvarplus {f}{x}{n+\beta} &:= (n+1)! \  \frac{ f( x+ \epsilon) - T_n (x, \epsilon) }{\epsilon^{n+\beta}}   \ecma   	\\
   		\fracvarmin {f}{x}{n+\beta} &:= (-1)^{n} (n+1)! \ \frac{T_n (x, - \epsilon) - f (x -\epsilon)  }{\epsilon^{n+\beta}}   \epnt
   	\end{align}
   	where  $T_n(x, \epsilon), \ n \in \fclass{N}{}$ is the usual Taylor polynomial 
   	\[ 
   	T_n(x, \epsilon)= f(x) + \sum\limits_{k=1}^{n} \frac{f^{(k)}(x)}{k!} \epsilon^k \ecma
   	\]
   	and $\epsilon >0$ and  $0 < \beta \leq 1 $ are real parameters.	 
   \end{definition}
   
   \begin{definition}[Mixed order velocity]
   	\label{def:frdiff2}
   	The \textsl{fractional velocity} of  mixed order $n+\beta$ of function $f(x) \in   \fclass{C}{n}$  be defined as
   	\begin{align}
   		\label{eq:fracdiff1}
   		\fdiffplus {f}{x}{n+\beta} &:=  \llim{\epsilon}{0}{  \fracvarplus {f}{x}{n+\beta}}   \ecma   	\\
   		\fdiffmin {f}{x}{n+\beta} &:=     \llim{\epsilon}{0}{\fracvarmin {f}{x}{n+\beta}   }   \ecma
   	\end{align}
   	where $\epsilon >0$ and  $0 < \alpha \leq 1 $ are real parameters.	 
   \end{definition}
   
   \begin{remark}
   	The mixed-order definition obviously applies to functions that are in general differentiable, i.e. they may be non differentiable only at isolated points. 
   	In the case of non-differentiable functions, the fractional variation should be taken only in strictly fractional order.  
   \end{remark}
   
%
   \begin{theorem}[Modular variation]
   	\label{th:modvar}
   	Let  $f(x) \in   \fclass{C}{n}$. 
   	Then the \textsl{fractional velocity} of  mixed order $n+\beta$  about  \textit{x} can be evaluated as 
   	\begin{align}
   		\label{eq:fracdiff3}
   		\fdiffplus {f}{x}{n+\beta} &=  \frac{(n+1)!}{{(n+\beta)}_{(n-1)} }
   		\llim{\epsilon}{0}{  	\frac{\deltaplus{f^{(n)}}{x} }{ \epsilon^{\beta}} }   \ecma   	\\
   		\fdiffmin {f}{x}{n+\beta} &=  \frac{(n+1)!}{{(n+\beta)}_{(n-1)} }
   		\llim{\epsilon}{0}{ \frac{ \deltamin{f^{(n)}}{x} }{\epsilon ^\beta}   }  \epnt
   	\end{align}
   	where  $\epsilon >0$ and  $0 < \beta \leq 1 $ are real parameters
   	and $ (n+\beta)_{(n)} $ is the decreasing Pochhammer symbol, 
   	\[ 
   	(n+\beta)_{(k)} = (n+\beta ) \ldots (n -k +\beta ) \epnt
   	\] 
   \end{theorem}
   Let's consider the point  $x \in   [x, x+\epsilon]$.
   \begin{proof}
   	\begin{description}
   		\item[Forward case]
   		The equation can be rearranged in the form
   		\[
   		f(x+ \epsilon) - T_n (x, \epsilon) = R_n(x) \ecma
   		\]
   		where $R_n(x) $ is the Taylor reminder. 
   		
   		Then we have to evaluate the limit
   		\[
   		\llim{\epsilon}{0}{}\frac{R_n(x)}{\epsilon^{n+\beta}} \epnt
   		\]
   		The evaluation of the limit can be performed by applying the l'H\^opital's rule \textit{n} times.
   		The Taylor reminder can be written in the form:
   		\[
   		R_n(x)  =\frac{1}{n!}\int\limits_{x}^{x+\epsilon} {\left(  \epsilon+x -t\right)  }^n\, 
   		\frac{{d}^{\, n+1}}{d{t}^{n+1}}\,f\left( t \right)   dt \epnt
   		\]
   		Then
   		\[
   		\left( R_n(x) \right) ^\prime_\epsilon = 
   		\frac{n}{n!}\,\int_{x}^{x+\epsilon}  \frac{{d}^{\, n+1}}{d{t}^{\, n+1}}f\left( t\right)   \,{\left( x-t+\epsilon\right) }^{n-1}dt = n \; R_{n-1} (x) \ecma
   		\]
   		where the index change in $R_n$ does not affect the derivative of $f(t)$.
   		
   		Therefore, the integral can be evaluated by induction to 
   		\[
   		R_1= \int_{x}^{x+\epsilon}\frac{{d}^{n+1}}{d{t}^{n+1}}\,f\left( t\right) dt =
   		\left.  f^{(n)}(t)\right|^{x+\epsilon}_x = \deltaplus{f^{(n)}}{x} 
   		\]
   		The denominator can be evaluated by \textit{n}-times differentiation about $\epsilon$ in a similar way to give
   		\[
   		\left( \epsilon^{n+\beta} \right) ^{(n)}_\epsilon = (n +\beta) (n -1 +\beta) \ldots \, (1+\beta)   \epsilon^{\beta} =    {(n+\beta)}_{(n-1)} \; \epsilon^{\beta}
   		\]
   		where we denote conveniently the decreasing Pochhammer symbol,
   		$x_{(n)}:= x \; (x-1) \ldots (x-n+1)$.
   		Therefore, finally
   		\[
   		\llim{\epsilon}{0}{}\frac{R_n(x)}{\epsilon^{n+\beta}} = \llim{\epsilon}{0}{}
   		\frac{\deltaplus{f^{(n)}}{x} }{{(n+\beta)}_{(n)}  \epsilon^{\beta}}
   		\]
   		
   		\item[Backward case]
   		The equation can be rearranged in the form
   		\[
   		T_n (x, -\epsilon) -f(x - \epsilon)  = - R_n(x) \epnt
   		\]
   		The Taylor reminder can be written in the form:
   		\[
   		R_n(x)  =  \frac{1}{n!}\int\limits_{x}^{x - \epsilon} {\left(  - \epsilon+x -t\right)  }^n\, 
   		\frac{{d}^{n+1}}{d{t}^{n+1}}\,\mathrm{f}\left( t \right)   dt
   		\]
   		Then we have to evaluate the limit
   		\[
   		\llim{\epsilon}{0}{} -\frac{R_n(x)}{\epsilon^{n+\beta}} \epnt
   		\]
   		Applying again  l'H\^opital's rule to the numerator yields 
   		\[
   		\left[  R_n(x) \right]^\prime_\epsilon = -n \; R_{n-1}(x)
   		\]
   		using the same convention. Therefore, 
   		\[
   		\left[ R_n(x) \right] ^{(n)}_\epsilon = - (-1)^{n} \deltamin{f^{(n)}}{x} \epnt
   		\]
   		The expression for the denominator is identical to the forward case.
   	\end{description}
   \end{proof}

   \begin{proposition}
   	\label{corr:mixed}
   	For a function $f(x) \in \holder{n+\beta}$ all mixed order fractional velocities exist up to the order $n+\beta$.
   \end{proposition}	
   Indeed, the proposition follows directly from Definition \ref{def:holder} for mixed order H\"older functions.
   
   \begin{corollary}[Continuous variation]
   	\label{th:fdiff1}
   	If $f^{(n)}(x)$  is continuous in the interval $x \in [x, x+ \epsilon]$ and
   	$f^{(n+1)}(x)$  is continuous in the interval $x \in (x, x+ \epsilon]$	then 
   	\begin{align}
   		\label{eq:fracdiff2}
   		\fdiffplus {f}{x}{n+\beta} &=  \frac{(n+1)!}{{(n+\beta)}_{(n)} } \llim{\epsilon}{0}{ \epsilon ^{1-\beta}  f ^{(n +1)}(x + \epsilon)  }   \ecma
   		\\
   		\fdiffmin {f}{x}{n+\beta} &=   \frac{(n+1)!}{{(n+\beta)}_{(n)} }  \llim{\epsilon}{0}{ \epsilon ^{1-\beta}  f ^{(n +1)}(x - \epsilon)  }   \epnt
   	\end{align}
   	where $ (n+\beta)_{(n)} $ denotes the decreasing Pochhammer symbol.
   \end{corollary}
   \begin{proof}
   	Previous work has established  \cite{Prodanov2015} that if the ordinary derivative 
   	$f^{\prime}(x)$  is continuous in the interval $x \in (x, \ x+ \epsilon]$ then the following equations hold
   	
   	\begin{flalign}
   		\label{eq:fdiffsimp}
   		\fdiffplus {f}{x}{\beta} =\frac{1}{\beta} \llim{\epsilon}{0}{ \epsilon ^{1-\beta}  f ^{\prime}(x + \epsilon)  }   \ecma
   		\\
   		\fdiffmin {f}{x}{\beta} = \frac{1}{\beta} \llim{\epsilon}{0}{ \epsilon ^{1-\beta}  f ^{\prime}(x - \epsilon)  }   \epnt
   	\end{flalign}
   	Application of l'H\^opital's rule to Theorem \ref{th:modvar} demonstrates the desired result.
   \end{proof}
   
 
\begin{corollary}
	\label{cor:unorder}
	Let $f(x) \in \fclass{C}{n} $ in the interval $ x \in [x, \; x+\epsilon ]$. 
	Then the non-integer order velocities below the critical order $n+\beta$ vanish.
\end{corollary}
\begin{proof}
	Let's calculate the fractional velocity of the mixed order $m +1\leq n$. By  Corollary \ref{th:fdiff1}
	\[
	\fdiffplus {f}{x}{m+\beta} =  \frac{(m+1)!}{{(m+\alpha)}_{(m)} } \llim{\epsilon}{0}{ \epsilon ^{1-\beta}  f ^{(m +1)}(x + \epsilon)  } \ecma
	\]
	where $\beta < 1$.
	Since by hypothesis $f ^{(n)}(x)$ is finite everywhere in the interval the last limit is zero.
	Therefore all fractional velocities of order $m+\beta < n$ vanish.
	The same argument can be applied to the backward fractional velocity. 
\end{proof}

\begin{remark}
	\label{rem:diff1}
	Following the notation of \cite{Cresson2005}, \cite{Cresson2011} we define the usual forward and backward differential maps as:
	\begin{flalign*}
	\diffplus{f}{x}{\epsilon} :=\frac{f(x + \epsilon)-f(x)}{\epsilon} \\
	\diffmin{f}{x}{\epsilon} :=\frac{f(x)- f(x - \epsilon)}{\epsilon}
	\end{flalign*}
	These operators can be considered as particular cases of the Fractal variation  according to Def. \ref{def:fracvar}.
\end{remark}
    \begin{theorem}[Fractional Rolle]
    	\label{th:frrolle}
    	Let $f(x) \in \holder{\beta}$ in $ \left[ a, b\right] $ and $f(a) = f(b)$ then
    	there exists a number $c \in \left[a, b \right]  $ such that 
    	$ \fracvarplus{f}{c}{\beta} \leq 0 $ and $ \fracvarmin{f}{c}{\beta} \geq 0 \,$.
    	Respectively,
    	$\fdiffmin{f}{c}{\beta}\leq 0$ and $\fdiffplus{f}{c}{\beta}\geq 0 \,$.
    	If both variations agree then
    	$\fdiffmin{f}{c}{\beta}  =\fdiffplus{f}{c}{\beta}=0 \,$.
    \end{theorem}
    \begin{proof}
    The proof of the theorem follows closely the proof of the generalized Rolle's Theorem.
    We can distinguish two cases:	
    	\begin{description}
    		\item[End-interval case] 
    		
    		Since $f(x)$ is continuous on $ \left[ a, b\right] $, it follows from the continuity property that $f(x)$ attains a maximum \textit{M} at some $c_1 \in  \left[ a, b\right] $ and a minimum \textit{m} at some $c_2 \in  \left[ a, b\right] $.
    		
    		Suppose $c_1$ and $c_2$ are both endpoints of $ \left[ a, b\right] $.
    		Because $f(a)=f(b)$ it follows that $m=M$ and so $f(x)$ is constant on $ \left[ a, b\right] $.
    		Hence $ \fracvarplus{f}{x}{\beta} = \fracvarmin{f}{x}{\beta} =0, \ \forall x \in \left[ a, b\right]$.
    		\item[Interior case]
    		Suppose then that the maximum is obtained at an interior point \textit{c} of $\left( a,b \right)$.
    		For an $\epsilon>0\,$, such that $c + \epsilon \in \left[ a, b\right]$  $f(c+\epsilon) \leq f(c)$ by assumption. 
    		Therefore, $ \fracvarplus{f}{c}{\beta} \leq 0$. 
    		Therefore, $\fdiffplus{f}{c}{\beta} \leq 0 \,$.
    		Similarly, for $c - \epsilon \in \left[ a, b\right]$ we have $\deltamin{f}{c} \geq 0 $ and
    		$ \fracvarmin{f}{c}{\beta} \geq 0$.
    		Therefore, $\fdiffmin{f}{c}{\beta} \geq 0 \,$.
    		Finally, if $\fdiffmin{f}{c}{\beta} =\fdiffplus{f}{c}{\beta}$ then both equal $0$.
    		The proof for the minimum follows identical reasoning. 
    	\end{description}
    \end{proof}

 \subsection{Bounds of variation}
  \label{sec:varbounds}
  \begin{lemma}[Bounds of forward variation]
   	\label{th:bondvar1}
   	Let $f(x) \in \holder{\alpha}$ in $[x, x+ \epsilon]$,  $\fdiffplus {f}{x}{\beta} \neq0$ and
   	$C_x$ and $C^{\prime}_{x}$ be constants such that 
   	i)  $C$ is the smallest number for which $|\Delta_{\epsilon}^{+}[f] (x)| \leq C \epsilon^\beta$ still holds, 	that is
   	$C_x =  \underset{C}{\inf} \{  |\Delta_{\epsilon}^{+}[f] (x)| \leq C \epsilon^\beta\}$ 	and 
   	ii) $C^{\prime}_x$ is the largest number $C$ for which  $ |\Delta_{\epsilon}^{+}[f] (x)| \geq C \epsilon^\beta$
   	still holds, that is 
   	$C^{\prime}_{x} =  \underset{C}{\sup} \{  |\Delta_{\epsilon}^{+}[f] (x) | \geq C \epsilon^\beta\}$
   	then
   	\[
   	|\fdiffplus {f}{x}{\beta} | = C_x =C^{\prime}_{x}
   	\]
   \end{lemma}
   \begin{remark}
   	We will assume that $\fdiffplus {f}{x}{\beta}$ is bounded in $[x, x+ \epsilon]$.
   \end{remark}
   \begin{proof}
   	We fix $x$ so that $C_x$ remains constant. Let $ \fdiffplus {f}{x}{\beta} =a$ for a certain bounded \textit{a}.
   	From the definition of a H\"{o}lder function it follows that $ |\Delta_{\epsilon}^{+}[f] (x) | \leq C_x \epsilon^\beta$ then we subtract $ a \epsilon^\beta $ from both sides of the inequality to obtain
   	\[
   	\left( C^{\prime}_{x} -a \right) \epsilon^\beta \leq	|\Delta_{\epsilon}^{+}[f] (x) | -a \epsilon^\beta \leq \left( C_x -a \right) \epsilon^\beta \epnt
   	\]
   	Further division by the positive quantity $  \epsilon^\beta$ results in 
   	\[
   	C^{\prime}_{x} -a  \leq \frac{|\Delta_{\epsilon}^{+}[f] (x) | -a \epsilon^\beta}{\epsilon^\beta} \leq   C_x -a  \epnt
   	\]
   	We have to further consider two cases.
   	If $\Delta_{\epsilon}^{+}[f] (x) \geq 0 $  then taking the limit both sides gives
   	\[
   	\llim{\epsilon}{0}{ \frac{\Delta_{\epsilon}^{+}[f] (x)  -a \epsilon^\beta}{\epsilon^\beta} } =0
   	\] 
   	and
   	$ \llim{\epsilon}{0}{ C_x -a} \geq 0$ but since $C_x$ is an infimum then $ \llim{\epsilon}{0}{ C_x  } = a$.
   	Since $C$ is constant with regard to $\epsilon$ it follows that $C_x =a$.
   	On the other hand, if $\Delta_{\epsilon}^{+}[f] (x) \leq 0 $ then
   	\[
   	\llim{\epsilon}{0}{ \frac{ - \Delta_{\epsilon}^{+}[f] (x)  -a \epsilon^\beta}{\epsilon^\beta} } = - 2 a
   	\] and
   	$ \llim{\epsilon}{0}{ C_x -a} \geq -2 a$, therefore $ \llim{\epsilon}{0}{ C_x } \geq -a$ and the same reasoning as in the previous case applies to yield $C_x= - a$.
   	Therefore, finally $C_x =|a|$.
   	The case for $C^{\prime}_{x}$ can be derived using identical reasoning. 
   \end{proof}

\begin{lemma}[Bounds of backward variation]
	\label{th:bondvar2}
	Let $f(x) \in \holder{\alpha}$ in $[x, x+ \epsilon]$,  $\fdiffmin {f}{x}{\beta} \neq0$ and
	$C_x$ and $C^{\prime}_{x}$ be constants such that 
	i)  $C$ is the smallest number $C$ for which $|\Delta_{\epsilon}^{-}[f] (x)| \leq C \epsilon^\beta$ still holds, 	that is
	$C_x =  \underset{C}{\inf} \{  |\Delta_{\epsilon}^{-}[f] (x)| \leq C \epsilon^\beta\}$ 	and 
	ii) $C^{\prime}_x$ is the largest number $C$ for which  $ |\Delta_{\epsilon}^{-}[f] (x)| \geq C \epsilon^\beta$
	still holds, that is 
	$C^{\prime}_{x} =  \underset{C}{\sup} \{  |\Delta_{\epsilon}^{-}[f] (x) | \geq C \epsilon^\beta\}$
	then
	\[
	|\fdiffmin {f}{x}{\beta} | = C_x =C^{\prime}_{x}
	\]
\end{lemma}

\begin{proposition} 
	\label{th:diffvar2}
	Let $f(x) \in \holder{\alpha}$ about $x$ and $0<\beta < \alpha \leq 1$ then
	$ \fdiffplus{f}{x}{\beta} =  \fdiffmin{f}{x}{\beta}   = 0
	$.
\end{proposition}   

\begin{lemma}
	\label{lem:LebesgueHolder}
	Let $f(x) \in \holder{\beta}$ in the interval $[a, b]$. 
	Then either $ \fdiffplus {f}{a}{\beta} \neq 0$ or 
	$ \fdiffplus {f}{b}{\beta} \neq 0$  or $ \fdiffmin {f}{a}{\beta} \neq 0$ or 
	$ \fdiffmin {f}{b}{\beta} \neq 0$.
\end{lemma}
\begin{proof}
	The proof follows from the   Lebesgue's differentiation Theorem\footnote{ 
		Let $f(x) : [a, b] \rightarrow \fclass{R}{} $ be an monotone function. Then $f(x)$ is differentiable almost everywhere in $[a, b]$.
		For an elementary proof see \cite{Faure2003}.
	}.
	The H\"older condition fulfills the assumption of the theorem.
	Therefore, the condition of Corollary \ref{th:fdiff1}, eq. \ref{eq:fdiffsimp} holds.
	Therefore, if $f^\prime(x)$ diverges at $a$ then $ \fdiffplus {f}{a}{\beta} \neq 0$ holds.
	Similar considerations are applicable in the other cases.
\end{proof}

\begin{theorem}[Discontinuous variation]
 	Let $f(x) \in \holder{\beta}$ in the interval $[x, x+ \epsilon]$. Then if $\fdiffplus {f}{x}{\beta} \neq 0$ then $\fdiffplus {f}{x}{\beta}$ is discontinuous about $x$. 
 	If $\fdiffmin {f}{x}{\beta} \neq 0$ then $\fdiffmin {f}{x}{\beta}$ is discontinuous about $x$. 
\end{theorem}
  \begin{proof}
	Suppose that  $\fdiffplus {f}{x}{\beta}  \in \holder{\alpha}$ and  $\fdiffplus {f}{x}{\beta} \neq 0$. 
	\begin{description}
		\item[Constant case]
			Suppose that $\fdiffplus {f}{x}{\beta} =K \neq 0$, constant. Then
			\[
			\llim{\epsilon}{0}{ \dfrac{f(x+ \epsilon) - f(x)}{\epsilon^\beta}} = K
			\]
			without loss of generality let's assume that $f(x_s)=0$ for a given fixed $x_s$. Therefore,
			\[
			\llim{\epsilon}{0}{ \dfrac{f(x_s+ \epsilon) }{\epsilon^\beta}} = K
			\]
			and $f(x +\epsilon)$ is of the form $g(x_s) \epsilon^\beta$. Therefore, $g(x)$ is constant. Therefore, by the reasoning of the previous case we conclude that for  $\beta <1$ the hypothesis is false.
		\item[Non-constant case]
		Suppose further that $\fdiffplus {f}{x}{\beta} \neq 0$.
		Then the conditions of Lemma \ref{lem:LebesgueHolder} hold.
		According to Corollary \ref{th:fdiff1} the co-domain of  $\fdiffplus {f}{x}{\beta} $ or respectively$ $\fdiffmin {f}{x}{\beta} is discrete. Therefore, so-derived functions are discontinuous. 
	\end{description}	
  \end{proof}
  
\section{Taylor expansions}
\label{sec:taylorexp}

\subsection{Fractional--order Taylor expansions}
\label{sec:taylorpoly}

\begin{theorem}[Fractional Taylor expansion]
	\label{th:holcomp1}
Let $f(x) \in \fclass{C}{0}$ about $x$ in the interval $ x \in [x, x+ \epsilon]$.
Then the existence of $\fdiffplus{f}{x}{\beta}$ at $x$ for a $\beta \leq 1$ implies that
\[
f(x +\epsilon)= f(x) + K \epsilon^\beta +\gamma \; \epsilon^\beta  \ecma
\] 
where $ K= \fdiffplus{f}{x}{\beta}$ and \llim{\epsilon}{0}{ \gamma}=0 and vice versa.
The existence of $\fdiffmin{f}{x}{\beta}$ implies that
\[
f(x -\epsilon)= f(x) - K^\star \epsilon^\beta +\gamma \; \epsilon^\beta  \ecma
\] 
where $ K^\star= \fdiffmin{f}{x}{\beta}$ and \llim{\epsilon}{0}{ \gamma}=0 and vice versa.
\end{theorem}
\begin{proof}
	Suppose that 
	\[
	f(x +\epsilon)= f(x) + K \epsilon^\beta +\gamma \; \epsilon^\beta  \ecma
	\] 
	where  \llim{\epsilon}{0}{ \gamma}=0.
	Then 
	$\fracvarplus{f}{x}{\beta} = K + \gamma $. Taking the limit provides the result.
	
	Suppose that 
	\[
	f(x - \epsilon)= f(x) - K^\star \epsilon^\beta +\gamma \; \epsilon^\beta  \ecma
	\] 
	where  \llim{\epsilon}{0}{ \gamma}=0.
	Then 
	$- \fracvarmin{f}{x}{\beta} = - K^\star + \gamma $. Taking the limit provides the result.

\end{proof}
Therefore, if follows that the following fractional Taylor expansion holds:
\begin{equation}
f(x \pm \epsilon) = f(x) \pm \fdiffpm {f}{x}{\alpha}   \epsilon^\alpha + \mathcal{O} \left(  \epsilon^{\alpha +} \right)  \ecma
\end{equation}
where the last expression in the notation is interpreted as an arbitrary exponent larger than $ \alpha$.

\subsection{Higher order fractional Taylor expansion}
\label{sec;taylor2}

\begin{definition}
	Let us define formal fractional Taylor polynomials for an  increasing  sequence $ \left\lbrace  \alpha_i  \right\rbrace  _{i=1}^n$  as:
	\begin{flalign*}
	T_{ \alpha, n}^{+}(x ,\epsilon) &= f(x) +  \sum\limits_{k=0}^{n} c_k   \epsilon^{\alpha +k }     \ecma \\
	T_{ \alpha, n}^{-}(x ,\epsilon) &= f(x) +  \sum\limits_{k=0}^{n} c_k  (- \epsilon)^{ k } \epsilon^{\alpha +k }\ecma
	\end{flalign*} 
	where $\alpha$ denotes the multi-index and   $ c_i  $ are arbitrary constants.
	\end{definition}
 We will establish the relationship to the growth of the function at $x\pm \epsilon$, respectively.
 We will look for $\mathcal{O}\left(  {\epsilon^m}\right) $ equivalence. That is
 
 \[
\llim{\epsilon}{0}{} \dfrac{ f(x+ \epsilon) - T_{ \alpha, n}^{+}(x ,\epsilon) }{\epsilon^m}  = 0
 \]
 Obviously for $k=0$ holds $c_0= \fdiffplus{f}{x}{\alpha}$. 
 Then the second coefficient can be calculated as
 \begin{flalign*}
 \llim{\epsilon}{0}{} \dfrac{ f(x+ \epsilon) - f(x) - \fdiffplus{f}{x}{\alpha}{\epsilon}^{\alpha}  }{\epsilon^m}  
 & = \llim{\epsilon}{0}{} \dfrac{ T_{ \alpha, n}^{+}(x ,\epsilon) - f(x) - \fdiffplus{f}{x}{\alpha}{\epsilon}^{\alpha}  }{\epsilon^m} \\
& = \llim{\epsilon}{0}{}\dfrac{  c_1 \epsilon^{\alpha +1 }  + \mathcal{O}(\epsilon^{2 +\alpha})   }{\epsilon^m} \\
& = \llim{\epsilon}{0}{ c_1 \epsilon^{\alpha +1 - m}}
 + \llim{\epsilon}{0}{\mathcal{O}(\epsilon^{2 +\alpha -m })}
 \end{flalign*}
 Therefore, in order for the RHS to be finite we must have $m=1+\alpha$.
 Then for the LHS we have
 \[
 \llim{\epsilon}{0}{} \dfrac{ f(x+ \epsilon) - f(x) - \fdiffplus{f}{x}{\alpha}{\epsilon}^{\alpha}  }{\epsilon^{1+\alpha}}
 = \frac{1}{{1+\alpha}} \llim{\epsilon}{0}{} \dfrac{ f^{\prime}(x+ \epsilon)  - \alpha \; \fdiffplus{f}{x}{\alpha}{\epsilon}^{\alpha-1}  }{\epsilon^{\alpha}}
 \]
The limit can be evaluated by application of L'H\^opital's rule and rationalization.
\begin{flalign*}
 \frac{1}{1+\alpha} \llim{\epsilon}{0}{} \dfrac{ f^{\prime}(x+ \epsilon)  - \alpha \; \fdiffplus{f}{x}{\alpha}{\epsilon}^{\alpha-1}  }{\epsilon^{\alpha}}  
  & = \frac{1}{1+\alpha}  \llim{\epsilon}{0}{} \dfrac{ \epsilon f^{\prime}(x+ \epsilon) - \alpha \; \fdiffplus{f}{x}{\alpha}{\epsilon}^{\alpha }}{\epsilon^{1 +\alpha}} \\
  &  = \frac{1}{1+\alpha}  \llim{\epsilon}{0}{}
  \dfrac{ \epsilon^{1-\alpha} f^{\prime}(x+ \epsilon) - \alpha \; \fdiffplus{f}{x}{\alpha} }{\epsilon} \\
  & = \frac{1}{1+\alpha} \llim{\epsilon}{0}{} \dfrac{\partial }{\partial \, \epsilon} \; \epsilon^{1-\alpha} f^{\prime}(x+ \epsilon)
\end{flalign*}
Therefore, 
\[
c_{1} = \frac{1}{1+\alpha} \llim{\epsilon}{0}{} \dfrac{\partial }{\partial \, \epsilon} \; \epsilon^{1-\alpha} f^{\prime}(x+ \epsilon) \epnt
\]

The same procedure can be extended for the general case by induction.
For an arbitrary $k \leq n $ we will have
\[
c_{k} =  \llim{\epsilon}{0}{}  \dfrac{ f(x+ \epsilon) - 	T_{ \alpha, n}^{+}(x ,\epsilon)}{\epsilon^{k + \alpha}} = \llim{\epsilon}{0}{}  \dfrac{ f(x+ \epsilon) - 	T_{ \alpha, k}^{+}(x ,\epsilon)}{\epsilon^{k + \alpha}}
\]
Therefore, we would have 
\[
 \frac{1}{k+\alpha} \llim{\epsilon}{0}{}\dfrac{ f^{\prime}(x+ \epsilon) - 	\left( T_{ \alpha, k}^{+}(x ,\epsilon) \right)^\prime_\epsilon }{\epsilon^{k + \alpha}}
\]
\[
 = \frac{1}{k+\alpha} \llim{\epsilon}{0}{} \dfrac{ \epsilon^{\alpha -1 }\left( 
 	\epsilon^{1-\alpha} f^{\prime}(x+ \epsilon)
 	- \sum\limits_{j=0}^{k} c_{j} (\alpha +j) \epsilon^{\alpha +j -1 -( \alpha -1) }
 	\right)  }{ \epsilon^{k + \alpha} }
\]
\[
 \frac{1}{k+\alpha}  \llim{\epsilon}{0}{} \dfrac{
 		\epsilon^{1-\alpha} f^{\prime}(x+ \epsilon)
 		- \sum c_{k} (\alpha +k) \epsilon^k }{ \epsilon^{k+1} }
\]
By applying $k$ times L'H\^opital's rule the denominator can be evaluated to give $ k!$ in order to eliminate the Taylor polynomial.
Therefore, finally
\[
c_{k} = \frac{1}{k! \, \left( 1+\alpha \right)} \llim{\epsilon}{0}{} \dfrac{\partial\, ^k }{\partial  \epsilon^k} \; \epsilon^{1-\alpha} f^{\prime}(x+ \epsilon) \epnt
\]
Applying similar procedure to the backward case would yield 
\[
c_{k} = \frac{ \left(-1 \right)^k}{k! \, \left( 1+\alpha \right)} \llim{\epsilon}{0}{} \dfrac{\partial\, ^k }{\partial  \epsilon^k} \; \epsilon^{1-\alpha} f^{\prime}(x - \epsilon) \epnt
\]

\begin{theorem}[Mixed order Taylor expansion]
\label{th:frtaylormixed}
	Let $f(x) \in \fclass{C}{n}$ about $x$ in the interval $ x \in (x, x + \epsilon]$ or 
	$[  x - \epsilon, x)$, respectively.
	Moreover, let  \fdiffplus{f}{x}{\alpha} or \fdiffmin{f}{x}{\alpha} exist and are non zero at $x$.
Then the following expansions hold:
 \[
  f(x +\epsilon) = f(x) +  \sum\limits_{k=0}^{n} c_k   \epsilon^{\alpha +k }  + \mathcal{O} \left( \epsilon^{n+\alpha}\right)  \ecma 
 \]
 where 
  \[
  c_{k} = \frac{1}{k! \, \left( k+\alpha \right)} \llim{\epsilon}{0}{} \dfrac{\partial\, ^k }{\partial  \epsilon^k} \; \epsilon^{1-\alpha} f^{\prime}(x+ \epsilon)
  \]
  and
   \[
   f(x - \epsilon) = f(x) +  \sum\limits_{k=0}^{n} c_k   \epsilon^{\alpha +k }  + \mathcal{O} \left( \epsilon^{n+\alpha}\right)  \ecma
   \]
   where
   \[
   c_{k} = \frac{ (-1)^k }{k! \, \left( k+\alpha \right)} \llim{\epsilon}{0}{} \dfrac{\partial\, ^k }{\partial  \epsilon^k} \; \epsilon^{1-\alpha} f^{\prime}(x - \epsilon) \epnt
   \]
\end{theorem}

\begin{figure}[h!]
	\centering
	\includegraphics[width=0.75\linewidth, bb=0 0 600 400 ]{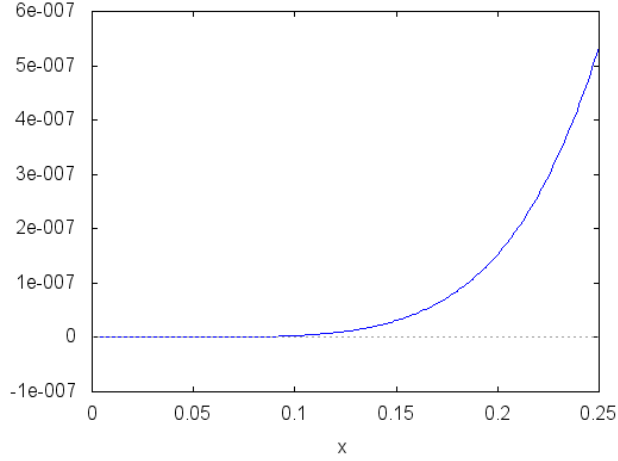}
	\caption{Error of the mixed-order Taylor expansion of $\arcsin(1-x)$ about the origin}
	\label{fig:example1}
\end{figure}
 	
 \begin{example}
 	Let's compute the fractional velocity of $f(x)= \arcsin(1-x)$ about $x=0$.
 	The first derivative of the function is
 	\[ 
 	f^\prime(x)= -\frac{1}{\sqrt{2 \; x-{{x}^{2}}}} \ecma
 	\]	
 	which is undefined about $0$.
 	Therefore, the function does not posses a Taylor expansion about 0.
 	On the other hand, 
 	\[
 	\fdiffplus {f}{0}{  \beta} =  -\dfrac{1}{\beta} \llim{\epsilon}{0}{
 		\frac{ \epsilon^{\beta - \frac{1}{2}}}{\sqrt{2 - \epsilon}} }
 	\]
 	Therefore, for $ \beta=1/2 \;  \fdiffplus{f}{0}{  \beta}  = - \sqrt{2} $. 
 	The fractional Taylor expansion about 0 then is 
 	\[
 	f(x) = \frac{\pi}{2} - \sqrt{2} \, \sqrt{x} + \mathcal{O}(\sqrt{x})
 	\]
 	
 	The regularized derivative then is
 	\[
 	\dj f(0) = \llim{\epsilon}{0}{} \dfrac{\arcsin( 1- \epsilon) - \dfrac{\pi}{2} + \sqrt{2} \, \sqrt{\epsilon}}{\epsilon} = 
 	\]
 	
 	\[
 	\llim{\epsilon}{0}{} \dfrac{1}{\sqrt{\epsilon}} \left( \frac{1}{\sqrt{2}} - \frac{1}{\sqrt{2-\epsilon}}  \right) = -\llim{\epsilon}{0}{}  \frac{   \sqrt{\epsilon} }{ \sqrt{2-\epsilon}^3 } =0
 	\]
 	Applying Theorem \ref{th:frtaylormixed} gives the following approximation:
 	
	\[
	f(x) \approx \frac{\pi }{2} -\sqrt{2}  \sqrt{x} -\frac{{{x}^{\frac{3}{2}}}}{3\cdot {{2}^{\frac{3}{2}}}}
	-\frac{3\cdot {{x}^{\frac{5}{2}}}}{5\cdot {{2}^{\frac{9}{2}}}}
	-\frac{5\cdot {{x}^{\frac{7}{2}}}}{7\cdot {{2}^{\frac{13}{2}}}}
	-\frac{35\cdot {{x}^{\frac{9}{2}}}}{9\cdot {{2}^{\frac{21}{2}}}} + \ldots
	\]
 	The approximation can be appreciated in Fig. \ref{fig:example1} where the error of this expansion is plotted.
 \end{example}

\section{Derivative regularization}
\label{sec:reg}
  
\subsection{Fractional order regularization}
\label{sec:fracreg}

 Here we will repeat  without proof  results that will be used further in the regularization procedure:
 \begin{proposition} 
 	\label{th:diffvar1}
 	Let $f(x) \in \fclass{C}{1}$ about $x$ and $0<\beta <1$ then
 	$ \fdiffplus{f}{x}{\beta} =  \fdiffmin{f}{x}{\beta}   = 0
 	$.
 \end{proposition}

 \begin{proposition}[Duality of the limit of variation]
 	\label{prop:dualvar}
 	If $f(x)$ is defined in $[ x - \epsilon, x+\epsilon]$,  the limit
 	\llim{\epsilon}{0}{\fracvarplus{f}{x}{\beta} } exists and
 	$
 	\llim{\epsilon}{0}{\dfrac{\Delta_{\epsilon} ^2 [f] (x)}{\epsilon^\beta} } = 0
 	$ then  $\fdiffplus {f}{x}{ \beta} = \fdiffmin {f}{x}{ \beta}$.
 	The proof is detailed in \cite{Prodanov2015}.  
 \end{proposition}
 
The value of the fractional velocity of a function $f(x)$ of fractional grade about a point can be regularized to zero using the following result:
  \begin{theorem}[Fractional variation regularization]
  	\label{th:regtaylor}
  Let  $f(x) \in \fclass{C}{0}$  and $\fdiffplus{f}{x}{\beta}$ and/or $\fdiffmin{f}{x}{\beta}$ be defined at $x$ for a given $0 < \beta < 1$  then
  \[
  	\llim{\epsilon}{0}{   \frac{f(x +\epsilon) - g(x, \epsilon)}{\epsilon^\beta}} = \llim{\epsilon}{0}{    \frac{g^\star(x, \epsilon) -  f(x -\epsilon) }{\epsilon^\beta}} =0
  \]
  	where $g(x, \epsilon)= f(x) + \fdiffplus{f}{x}{\beta} \; \epsilon^\beta$ and
  $g^\star(x, \epsilon)= f(x) - \fdiffmin{f}{x}{\beta} \; \epsilon^\beta$ are   regularizing kernels.
  \end{theorem}
  \begin{proof}
  	Let $g(x, \epsilon)$ be smooth about the  argument $\epsilon$. 
  	We will consider further four cases.
  	\begin{description}  
  	\item[ Differentiable $f(x)$]
  
  	\textit{Forward case} :
  	Let $f(x) \in \fclass{C}{1}$ in the left-open interval $ (x, x+ \epsilon]$ and $g(x, y) \in \fclass{C}{1} $ by its second variable in the closed interval  $ y \in [0,  \epsilon]$.
  	Then we can apply l'H\^opital's rule to the limit 
  	\begin{flalign*}
  	 	\llim{\epsilon}{0}{   \frac{f(x +\epsilon) - g(x, \epsilon)}{\epsilon^\beta}} & = 
  	 	\llim{\epsilon}{0}{ } \frac{\epsilon^{1- \beta}}{\beta} \left(  f^{\prime}(x +\epsilon)  - g^{\prime}_\epsilon (x, \epsilon)\right) \\
  	 	& = \fdiffplus{f}{x}{\beta} - \frac{1}{\beta} \llim{\epsilon}{0} { \epsilon^{1- \beta} g^{\prime}_\epsilon (x, \epsilon) } 
  	\end{flalign*}
 
  	by Th. \ref{corr:mixed}.
  	Let's suppose that $ \frac{1}{\beta} \epsilon^{1-\beta} g^{\prime}_\epsilon (x, \epsilon) = K $.
  	Then the resulting ordinary differential equation can be solved to yield for $g(x, \epsilon)$:
  	\[
  	g(x, \epsilon) = C + K \epsilon^\beta \epnt
  	\]
  	Therefore, $ K= \fdiffplus{f}{x}{\beta}$. 
  	Moreover, setting $C = f(x)$ allows for interpolation with the fractional variation operator.
  	
  	\textit{Backward case}:
  	\begin{flalign*}
  	 \llim{\epsilon}{0}{ \frac{ g(x, \epsilon)- f(x -\epsilon) }{\epsilon^\beta}} 
    & =	\llim{\epsilon}{0}{ } \frac{\epsilon^{1- \beta}}{\beta} \left(  f^{\prime}(x -\epsilon)  + g^{\prime}_\epsilon (x, \epsilon)\right) \\
    & = \fdiffmin{f}{x}{\beta} + \frac{1}{\beta} \llim{\epsilon}{0} { \epsilon^{1- \beta} g^{\prime}_\epsilon (x, \epsilon) }  
  	\end{flalign*}
  	by Th. \ref{corr:mixed}.
  	Therefore, in this case $ K= - \fdiffmin{f}{x}{\beta}$. 

  	\item[ Non-differentiable $f(x)$]
  	
  	\textit{Forward case}:
  	According to Th. \ref{th:holcomp1} the existence of $\fdiffplus{f}{x}{\beta}$ implies that
  	\[
  	f(x +\epsilon)= f(x) + K \epsilon^\beta +\gamma \; \epsilon^\beta  \ecma
  	\] where $ K= \fdiffplus{f}{x}{\beta}$ and \llim{\epsilon}{0}{ \gamma}=0.
  	Therefore,
  	\begin{flalign*}
  	\llim{\epsilon}{0}{   \frac{f(x +\epsilon) - g(x, \epsilon)}{\epsilon^\beta}} & =  
  	\llim{\epsilon}{0}{ \frac{f(x) + K \epsilon^\beta +\gamma \; \epsilon^\beta - g(x,  \epsilon)}{\epsilon^\beta }} \\
  	& = \llim{\epsilon}{0}{\frac{f(x) + K \epsilon^\beta  - g(x,  \epsilon)}{\epsilon^\beta }}
	\end{flalign*}
  	Therefore, since the value $f(x)$ is arbitrary we can set 
  	$g(x,\epsilon) = f(x)+ K \epsilon^\beta $.
  	
  	\textit{Backward case}:
 	\[
 	f(x -\epsilon)= f(x) - K \epsilon^\beta +\gamma \; \epsilon^\beta  \ecma
 	\] where $ K= \fdiffplus{f}{x}{\beta}$ and \llim{\epsilon}{0}{ \gamma}=0.
 	Therefore,
 	\begin{flalign*}
 	\llim{\epsilon}{0}{  \frac{ g(x, \epsilon)-  f(x -\epsilon)  }{\epsilon^\beta}} & = 
 	\llim{\epsilon}{0}{ \frac{- f(x) + K \epsilon^\beta -\gamma \; \epsilon^\beta + g(x,  \epsilon)}{\epsilon^\beta }}   \\
 	& = \llim{\epsilon}{0}{\frac{- f(x) + K \epsilon^\beta  + g(x,  \epsilon)}{\epsilon^\beta }}
 	\end{flalign*}
 	Therefore, in this case 
 	$g(x,\epsilon) = f(x)- K \epsilon^\beta $.
 	\end{description}
  \end{proof}
  
  \begin{remark}
  	The proof can be extended also for $\beta=1$ (Differentiable case) where it can be noted that 
  	$ g^\prime_\epsilon (x, \epsilon) = K$ and therefore  
  	$ g (x, \epsilon) = K \; \epsilon + C $ implying that $ g (x, \epsilon) = f(x) + f^\prime (x) \; \epsilon  $.
  \end{remark}

 The last Theorem can be also used to extend the usual l'H\^opital -- Bernoulli rule for the indefinite form 
 $0/0$:
  \begin{theorem}[Fractional l'H\^opital -- Bernoulli, case 0/0]
  	\label{th:frLhop}
  	Let $f(x) =0$ and $g(x) =0$ for a given $x$ for two functions $f,g \in \fclass{C}{0}$ 
  	and   $\fdiffplus{g}{x}{\beta} \neq 0$ identically in $[x, x+\epsilon]$, 
  	Then  
  	\begin{flalign*}
  	 	 \frac{\fdiffplus{f}{x}{\beta} }{\fdiffplus{g}{x}{\beta}} =
  	 	\llim{\epsilon}{0}{ \frac {f(x+ \epsilon)} {g(x+ \epsilon)} } \\
  	 	 \frac{\fdiffmin{f}{x}{\beta} }{\fdiffmin{g}{x}{\beta}}=
  	 	\llim{\epsilon}{0}{ \frac {f(x- \epsilon)} {g(x- \epsilon)} } 
  	\end{flalign*}
  \end{theorem}
  \begin{proof}
  	Forward case:
  	 By assumption
  	 $ f(x + \epsilon) = f(x) + K \epsilon^\beta +\gamma \epsilon^\beta =  K \epsilon^\beta +\gamma \epsilon^\beta
  	 $ setting $ \fdiffplus{f}{x}{\beta} = K$.
  	 $ g(x + \epsilon) = g(x) + P \epsilon^\beta +\gamma \epsilon^\beta =  P \epsilon^\beta +\gamma^\star \epsilon^\beta
  	 $ setting $ \fdiffplus{g}{x}{\beta} = P$.
  	 Then we have
  	 \[
  	 \frac{f(x + \epsilon)}{g(x + \epsilon)} = \frac{K + \gamma}{P + \gamma^\star} \epnt
  	 \]
  	 Taking the limit gives 
  	 \[
  	 	\llim{\epsilon}{0}{ \frac {f(x+ \epsilon)} {g(x+ \epsilon)} }  = \frac{K}{P} = \frac{\fdiffplus{f}{x}{\beta} }{\fdiffplus{g}{x}{\beta}}
  	 \]
  	 The same reasoning can be applied for the backward case.
  \end{proof}

  \begin{definition}[Composition of variations]
  	\label{def:comp}
  	Let $f(x)$ be continuous about the point $x$. 
  	Let the composition of fractal variations be defined as
  	\begin{align}
  	\label{eq:fraccomp}
  	\fdiffplus {f}{x}{\alpha   \circ \beta} := \llim{\epsilon}{0}{\frac{  \Delta_{\epsilon}^{+}[f] (x) - \fdiffplus {f}{x}{\beta} \epsilon^{\beta} }{\epsilon ^{\alpha+ \beta} }}   \ecma
  	\\
  	\fdiffmin {f}{x}{\alpha  \circ \beta} := \llim{\epsilon}{0}{\frac{  \Delta_{\epsilon}^{-}[f] (x) - \fdiffmin {f}{x}{\beta} \epsilon^{\beta} }{\epsilon ^{\alpha+ \beta}} }  \ecma
  	\end{align}
  	where $ 0 < \alpha, \beta \leq 1$.
  	
  \end{definition}
  \begin{remark}
  	We will assume that $\fdiffplus {f}{x}{0}   \equiv 0$. Indeed, since $f(x)$ is continuous about $x$ then
  	$\llim{\epsilon}{0}{ \Delta_{\epsilon}^{+}[f] (x)  ]} = 	\llim{\epsilon}{0}{ \Delta_{\epsilon}^{-}[f] (x)  ]} =0$.
  \end{remark}
  \begin{proposition}
  	Under this definition for the general case $\alpha \neq \beta$ it follows that
  	$ \fdiffplus {f}{x}{\alpha   \circ \beta} \neq \fdiffplus {f}{x}{\beta   \circ \alpha} $.
  \end{proposition}
  	\begin{proof}
	  	Indeed if we take the difference
	  	\[
	  	\frac{  \Delta_{\epsilon}^{+}[f] (x) - \fdiffplus {f}{x}{\beta} \epsilon^{\beta} }{\epsilon ^{\alpha+ \beta}} - \frac{  \Delta_{\epsilon}^{+}[f] (x) - \fdiffplus {f}{x}{\alpha} \epsilon^{\alpha} }{\epsilon ^{\alpha+ \beta} } = - \frac{ \fdiffplus {f}{x}{\beta} \epsilon^{\beta}   - \fdiffplus {f}{x}{\alpha} \epsilon^{\alpha} }{\epsilon ^{\alpha+ \beta}}
	  	\]
  		If we take the limit and apply  L'H\^opital's rule for $\epsilon$ we get
  		\[
  		-\llim{\epsilon}{0}{} \frac{1}{\alpha + \beta} \epsilon^{1- \left( \alpha + \beta\right) } 
  		\left(
  		\beta \fdiffplus {f}{x}{\beta} \epsilon^{\beta -1} - \alpha \fdiffplus {f}{x}{\alpha} \epsilon^{\alpha -1} \right) \]
  		\[
  		= - \frac{1}{\alpha + \beta} \llim{\epsilon}{0}{} 
  		\left( \frac{\beta}{\epsilon^{\alpha}} \fdiffplus {f}{x}{\beta} - 
  		\frac{\alpha}{\epsilon^{\beta}} \fdiffplus {f}{x}{\alpha} 
  		\right) 
  		\]
  		Therefore, if $  \alpha  \neq  \beta $ the limit is diverging.
  	\end{proof}

  \begin{definition}
  	\label{def:regul}
  	Fractional order  regularization is defined in the following way:
  	\begin{flalign*}
  	\frac{  \dj \, ^{\beta+} }{ d x} f(x):=  \fdiffplus {f}{x}{1-\beta   \circ \beta} = \llim{\epsilon}{0}{\frac{  \Delta_{\epsilon}^{+}[f] (x) - \fdiffplus {f}{x}{\beta} \epsilon^{\beta} }{\epsilon  }} \\
  	\frac{  \dj \, ^{\beta-} }{ d x} f(x):=  \fdiffmin {f}{x}{1-\beta   \circ \beta} = \llim{\epsilon}{0}{\frac{  \Delta_{\epsilon}^{-}[f] (x) - \fdiffmin {f}{x}{\beta} \epsilon^{\beta} }{\epsilon  }}
  	\end{flalign*}
  	We will require as usual that the forward and backward regularized derivatives be equal for a uniformly continuous function. 
  We will use further also a shortened notation $\dj f (x)$ when the value of $\beta$ is fixed. 
  \end{definition}

   By direct calculation we can use the composition of variations to regularize the first derivative about non-differentiable points. 	The proof follows from Th. \ref{th:regtaylor}.

  Having established the bounds of variation we can proceed to a new result for a H\"olderian function
  of grade $\alpha$.      	
  Under this regularization we have the following result:
  \begin{theorem}[Orthogonality of grades]
  	Let $f \in \holder{\alpha}$, which is not a sum of \fclass{C}{1} and \holder{\alpha} functions. Then 
  	\[	
  	\dj^{\alpha} \, f (x) \ \fdiffplus {f}{x}{\alpha}  =
  	\dj^{\alpha} \, f (x) \ \fdiffmin {f}{x}{\alpha} = 0 \ecma
  	\]
  	where we interpret regularization with regard to the grade $\alpha$.
  \end{theorem}
  \begin{proof}
  	
  	Let $f(x) \in \holder{\alpha}$.
  	We fix $x$ so that $C_x$ remains constant. 
  	\[
  	C_x^\prime \epsilon^\alpha \leq | \Delta^{+}_\epsilon f(x) | \leq  C_x   \epsilon^\alpha 
  	\]
  	Let $ \fdiffplus {f}{x}{\alpha} =a$ assuming that \textit{a} is bounded.
  	Subtracting $ |a| \epsilon^\alpha$ from both sides gives
  	\[
  	( C_x^\prime -|a|) \epsilon^\alpha \leq |\Delta^{+}_\epsilon f(x) | - |a| \epsilon^\alpha  \leq ( C_x - |a| ) \epsilon^\alpha 
  	\]
  	Therefore, if we take the \textsl{supremum} and the \textsl{infimum} bounds on $C^\prime$ and $C$ respectively we will get
  	\[
  	|\Delta^{+}_\epsilon f(x) | - |a| \epsilon^\alpha = 0
  	\]
  	There are two cases to consider.
  	If $a \geq 0$ implying also $\Delta^{+}_\epsilon f(x) \geq 0$ then
  	
  	\[
  	\Delta^{+}_\epsilon f(x)  - \fdiffplus {f}{x}{\alpha}  \epsilon^\alpha =0
  	\]
  	If $a \leq 0$ implying also $\Delta^{+}_\epsilon f(x) \leq 0$ then
  	\[
  	-\Delta^{+}_\epsilon f(x)  + \fdiffplus {f}{x}{\alpha}  \epsilon^\alpha =0
  	\]
  	Therefore in both cases 
  	\[ \dfrac{\Delta^{+}_\epsilon f(x)  - \fdiffplus {f}{x}{\alpha}  \epsilon^\alpha}{\epsilon} =0 \ecma
  	\]
  	but this is the definition of the regularized derivative. Therefore, if follows that 
  	\[
  	\frac{\ \dj^\alpha}{dx} f(x) = 0 \epnt
  	\]
  \end{proof}
  \begin{example}
  	A key aspect to the former proof was the requirement ot keep $C_x$ remains constant. We will see that its relaxation invalidates  part of the result.
  	Let $f(x) = x + \sqrt{x}$.
  	Then $\fdiffplus{f}{0}{1/2} = 1$.
  	\[
  	\deltaplus{f}{\epsilon} - \sqrt{\epsilon} = \epsilon
  	\]
  	Therefore, $ \dj f (0)=1$. Inspection of the equations shows that $C_x$ is a function of the increment
  	$\epsilon$. Therefore, the hypothesis can not be relaxed.
  \end{example}
\subsection{Mixed order regularization}  
\label{sec:mixedreg}

The regularization procedure can be extended also towards mixed orders in the following way:
\begin{definition}
	Let us define formal fractional Taylor polynomials for an  increasing positive sequence $ \left\lbrace  \alpha_i  \right\rbrace  _{i=1}^n$ bounded from above by 1  as:
	\begin{flalign*}
	T_{ \alpha, n}^{+}(x ,\epsilon) &= f(x) +  \sum\limits_{\alpha_i}^{} c_i   \epsilon^{\alpha_i }   +
	\sum\limits_{k=1}^{n} \dfrac{f^{(k)}(x )}{k!}  \epsilon^{ k }  \ecma \\
	T_{ \alpha, n}^{-}(x ,\epsilon) &= f(x) -  \sum\limits_{\alpha_i}^{} c^\star_i    \epsilon^{\alpha_i } +
	\sum\limits_{k=1}^{n} \dfrac{f^{(k)}(x )}{k!}  (- \epsilon)^{ k } \ecma
	\end{flalign*} 
	where $\alpha$ denotes the multi-index and   $ c_i  $ are arbitrary constants.
	Then $\alpha$-regularized derivatives are defined as
	\begin{flalign*}
	\dfrac{\dj^{\;n+}}{dx^n} f(x) &= (n+1)!  \llim{\epsilon}{0}{}
	\  \frac{ f( x+ \epsilon) - T_{\alpha, n} (x, \epsilon) }{\epsilon^{n+\beta}} \ecma \\
	\dfrac{\dj^{\;n-}}{dx^n} f(x) &= (-1)^n (n+1)!  \llim{\epsilon}{0}{}
	\  \frac{ T_{\alpha, n}^{-} (x, \epsilon) - f( x - \epsilon) }{\epsilon^{n+\beta}} \ecma
	\end{flalign*} 
	where $\beta = \sup {\alpha_i}$. We will require as usually that the forward and backward regularized derivatives be equal for a uniformly continuous function. 
\end{definition}
\begin{proposition}
	For a multi-H\"older function  \holder{  \alpha_1 + \alpha_2 + \ldots} according to Th. \ref{th:regtaylor} we have
	$c_i = \fdiffplus{f}{x}{\alpha_i} $ and 
	$c^\star_i = \fdiffmin{f}{x}{\alpha_i} $.
\end{proposition}
The proof follows by induction considering that the sequence $  \left\lbrace  \alpha_i  \right\rbrace  $ is increasing. 
Then according to Th. \ref{th:regtaylor} a new regularization term is added for every $\alpha_i$.

\begin{proposition}
\label{prop:taylorcomp}
 
Let us consider the following fractional Taylor polynomials for a positive  increasing regular sequence $ \left\lbrace  \alpha_i  \right\rbrace  _{i=1}^n$ :
\begin{flalign*}
T_{ \alpha, n}^{+}(x ,\epsilon) &= f(x) +  \sum\limits_{k=1}^{n} c_k   \epsilon^{\alpha k }     \ecma \\
T_{ \alpha, n}^{-}(x ,\epsilon) &= f(x) +  \sum\limits_{k=1}^{n} c_k  (- 1)^{ k } \epsilon^{\alpha  k }\ecma
\end{flalign*} 
where $\alpha$ denotes the multi-index and   $ c_i  $ are arbitrary constants.
Then the following expansion holds for a compound function $ f(x^\alpha)$, $\alpha \in (0, 1)$:
  \[
  f(x +\epsilon) = f(x) +  \sum\limits_{k=1}^{n} c_k   \epsilon^{\alpha k }  + \mathcal{O} \left( \epsilon^{n \,\alpha}\right)  \ecma 
  \]
  for
  \[
  c_{k} = \frac{1}{k!} \left. \dfrac{\partial\, ^k }{\partial  u^k} \;    f (u) \right|_{u={x}^{\alpha}}
  \]
  
  and
  
  \[
  f(x - \epsilon) = f(x) +  \sum\limits_{k=1}^{n} c_k   \epsilon^{\alpha k }  + \mathcal{O} \left( \epsilon^{n\,\alpha}\right)  \ecma
  \]
  for
  \[
  c_{k} = \frac{(-1)^k}{k!} \left. \dfrac{\partial\, ^k }{\partial  u^k} \;  f (u) \right|_{u={x}^{\alpha}} \epnt
  \]
\end{proposition}
 \begin{proof}
 	  We will establish the relationship to the growth of the function at $x\pm \epsilon$, respectively.
 	  The technique of the proof is similar to the previous case.
 	  
 	  Obviously for $k=1$ holds $c_1= \fdiffplus{f}{x}{\alpha}$. 
 	  Then the second coefficient can be calculated as
 	  \begin{flalign*}
 	  \llim{\epsilon}{0}{} \dfrac{ f(x+ \epsilon) - f(x) - \fdiffplus{f}{x}{\alpha}{\epsilon}^{\alpha}  }{\epsilon^m}  
 	  & = \llim{\epsilon}{0}{} \dfrac{ T_{ \alpha, n}^{+}(x ,\epsilon) - f(x) - \fdiffplus{f}{x}{\alpha}{\epsilon}^{\alpha}  }{\epsilon^m} \\
 	  & = \llim{\epsilon}{0}{}\dfrac{  c_1 \epsilon^{2\, \alpha   }  + \mathcal{O}(\epsilon^{2 \, \alpha})   }{\epsilon^m} \\
 	  & = \llim{\epsilon}{0}{ c_1 \epsilon^{2 \alpha   - m}}
 	  + \llim{\epsilon}{0}{\mathcal{O}(\epsilon^{2 \, \alpha -m })}
 	  \end{flalign*}
 	  Therefore, in order for the RHS to be finite we must have $m=2 \alpha$.
 	  Then for the LHS we have
 	  \[
 	  \llim{\epsilon}{0}{} \dfrac{ f(x+ \epsilon) - f(x) - \fdiffplus{f}{x}{\alpha}{\epsilon}^{\alpha}  }{\epsilon^{2\,\alpha}}
 	  = \frac{1}{{2 \, \alpha}} \llim{\epsilon}{0}{} \dfrac{ f^{\prime}(x+ \epsilon)  - \alpha \; \fdiffplus{f}{x}{\alpha}{\epsilon}^{  \alpha-1}  }{\epsilon^{2\, \alpha - 1}}
 	  \]
 	  The argument of last limit is then
 	  \[
 	  \dfrac{\epsilon^{  \alpha-1 }} {\epsilon^{2\, \alpha - 1}}
 	  \left( \epsilon^{ 1- \alpha  }  f^{\prime}(x+ \epsilon) - \fdiffplus{f}{x}{\alpha} \right) 
 	  = \dfrac{\epsilon^{ 1- \alpha  }  f^{\prime}(x+ \epsilon) - \fdiffplus{f}{x}{\alpha} }{\epsilon^\alpha  }
 	  \]
 	  
 	  The limit can be evaluated by application of L'H\^opital's rule and rationalization.
 	  \[
 	  \llim{\epsilon}{0}{} \dfrac{\epsilon^{ 1- \alpha  }  f^{\prime}(x+ \epsilon) - \fdiffplus{f}{x}{\alpha} }{\epsilon^\alpha  } = \frac{1}{\alpha}  \llim{\epsilon}{0}{} \epsilon^{ 1- \alpha  } 
 	  \frac{\partial}{ \partial \epsilon} \left( \epsilon^{ 1- \alpha  }  f^{\prime}(x+ \epsilon)\right) 
 	  \]
 	  Therefore,
 	  \[
 	  c_{2} = \dfrac{1}{2 \, \alpha^2} \llim{\epsilon}{0}{} \epsilon^{ 1- \alpha  } 
 	  \frac{\partial}{ \partial \epsilon} \left( \epsilon^{ 1- \alpha  }  \frac{\partial}{ \partial \epsilon}  f (x+ \epsilon)\right) 
 	  \]
 	  However, this can be recognized as double application of the ordinary compound differential rule.
 	  Indeed,
 	  let us suppose that 
 	  \[
 	  f(x) = f(u(x)) \ecma \ u(x)=x^\alpha
 	  \]
 	  Then
 	  \[
 	  \left(  \frac{d  u(x)}{dx} \right)^{-1} \frac{d \, f(x)}{dx} = \frac{\partial    f(x)}{\partial u}
 	  \]
 	  and
 	  \[
 	  \frac{\partial   f(x)}{\partial u} = \frac{x^{1-\alpha}}{\alpha}  \frac{d \, f(x)}{dx}
 	  \]
 	  and we can recognize the derivative of $x^\alpha $ evaluated at $x=\epsilon$.
 	  Similar arguments can be demonstrated to hold  for all $k \leq n$ by induction.

 	  Therefore, in the general case we would have as expected expansion of the function $f(x)$ in Taylor series w.r.t $u (x) = x^\alpha$ followed by substitution.
 	  Therefore,
 	  \[
 	  f(x +\epsilon) = f(x) +  \sum\limits_{k=1}^{n} c_k   \epsilon^{\alpha k }  + \mathcal{O} \left( \epsilon^{n \,\alpha}\right)  
 	  \]
 	  for
 	  \[
 	  c_{k} = \frac{1}{k!} \dfrac{\partial\, ^k }{\partial  u^k} \;  f (u)
 	  \]
 	  and
 	  \[
 	  f(x - \epsilon) = f(x) +  \sum\limits_{k=1}^{n} c_k   \epsilon^{\alpha k }  + \mathcal{O} \left( \epsilon^{n\,\alpha}\right)   
 	  \]
 	  for
 	  \[
 	  c_{k} = \frac{(-1)^k}{k!}  \dfrac{\partial\, ^k }{\partial  u^k} \;  f (u) \epnt
 	  \]

 \end{proof}
\begin{remark}
	  	  In the original notation we can also state the result explicitly
	  	  \begin{equation}
	  	  c_{k} = \frac{\left(  \pm 1 \right)^k }{k! \ \alpha ^k } \ \llim{\epsilon}{0}{} \prod_{}^{k } \circ \left(  \epsilon^ {1-\alpha}  \dfrac{\partial\  }{\partial  \epsilon} \right)  f (x \pm \epsilon)
	  	  \end{equation}
	  	  where the product has to be understood as composition.
\end{remark}

   \begin{example}
   Consider the function $f(x)=\cos (x^{1/3})$. We will develop the fractional Taylor expansion about $x=0$.
   The first derivative of the function is
   \[ 
   f^\prime(x)=  - \frac{1}{3}\frac{\sin\left( {{x}^{\frac{1}{3}}}\right) }{  {{x}^{\frac{2}{3}}}} \ecma
   \]
   which is undefined for $x=0$.
   According to Th. \ref{th:frtaylormixed} for $\alpha=\frac{1}{3}$ we would have

     \[
     c_1 = 
     - \lim_{\epsilon \to 0}{ \sin   {{\epsilon}^{1/3}}   } = 0
     \]
   Continuing for other degrees
    \[
    c_2 = 
    -\frac{1}{2}\lim_{\epsilon \to 0}{  \cos   {\epsilon}^{1/3} } = -\frac{1}{2}
    \]
    and so on.
    Therefore, in agreement with the previous proposition we would have
    \[
    f(x) \approx 
    1 - \frac{1}{2}\cdot {{x}^{\frac{2}{3}}} + \frac{1}{24}\cdot {{x}^{\frac{4}{3}}} 
     - \frac{1}{720}\cdot {{x}^{2}}+ \frac{1}{40320}\cdot {{x}^{\frac{8}{3}}} + \ldots
    \]
    \end{example}

\section{Compound differential rules}
\label{sec:comprules}

\begin{theorem}
	Let Let $ f(t, w ) \in \mathbb{C}^1$  be composition with
	$w(x) \in \holder{\beta}$ in the interval $ [x, x+\epsilon]$   then the following  expansions  hold: 
	
	\begin{flalign*}
	\fdiffplus{f}{w}{\beta}=
	\frac{\partial f}{\partial w } \fdiffplus{w}{x}{\beta} \\
	\fdiffmin{f}{w}{\beta}=
	\frac{\partial f}{\partial w } \fdiffmin{w}{x}{\beta}
	\end{flalign*}
\end{theorem}

\begin{proof}
	According to \cite{Prodanov2015} we have
	\[
	\fracvarplus{f}{w}{\beta}=
	\frac{\Delta^{+}_{\epsilon} f (w)}{\Delta w } \fracvarplus{w}{x}{\beta} \epnt
	\]
	Taking the limit provides the result.
	The backward case is proven in a similar way.
\end{proof}

\subsection{Regularized Taylor -- It\^o expansion}
\label{sec:itorule}
The  differentiation rule for compositions of functions can be derived using the regularization of the derivatives.
It provides a result formally analogous to  It\^o's Lemma. 
In a similar way the rule can be derived using a non-differentiable embedding principle  \cite{Cresson2011}.

\begin{theorem}[Taylor -- It\^o  expansion]
	\label{th:itotaulor}
	Let Let $ f(t, w ) \in \mathbb{C}^2$ be composition with $w(x) \in \holder{1/2}$ in the interval $x \in [x, x+\epsilon] $ then the following  expansions  hold: 
	\begin{flalign*}
	\frac{\dj^{+} }{dt} f(t,w) =\frac{\partial f}{\partial t} + \frac{\partial f}{\partial w}   \cdot   \frac{\dj^{+} }{dt } w (x)    + \frac{1}{2}   \frac{\partial ^2 f}{\partial w^2} [ w, w]^{+} (x)   \\
	\frac{\dj^{-} }{dt} f(t,w) =\frac{\partial f}{\partial t} + \frac{\partial f}{\partial w}   \cdot   \frac{\dj^{-} }{dt } w (x)    - \frac{1}{2}   \frac{\partial ^2 f}{\partial w^2} [ w, w]^{-} (x)
	\end{flalign*} 
\end{theorem}
\begin{proof}
We will prove first the forward case.

\begin{description}
\item [Forward case]
 By second order  Taylor's expansion we get
  \begin{flalign*}
  f( w(x + \epsilon), x + \epsilon)  = & f(w, t) + \frac{\partial f}{\partial x}   \; \epsilon    +
  \frac{\partial f}{\partial w} \deltaplus{w}{x} + \\
  & \frac{1}{2} \frac{\partial^2 f}{\partial x^2}   \; \epsilon^2 + 
  \frac{\partial^2 f}{\partial x \, \partial w  } (x) \deltaplus{w}{x} \epsilon + \\
  & \frac{1}{2} \frac{\partial^2 f}{\partial w^2}  \; \left( \deltaplus{w}{x}\right) ^2 
  + \mathcal{O } (\epsilon^2 )
  \end{flalign*}
 
  It follows that two cases have to be considered:
  
 \textbf{Differentiable case} 
 
This case is realized when $ \fdiffplus{w}{x}{1/2} =0$ and the ordinary derivative is defined:
 \begin{flalign*}
   \diffplus{f}{w, x }{\epsilon}  = & \frac{\partial f}{\partial x} + \frac{\partial f}{\partial w} \diffplus{w}{ x }{\epsilon} + \\
    & \frac{1}{2} \frac{\partial^2 f}{\partial x^2} \;\epsilon +
     \frac{\partial^2 f}{\partial x \, \partial w  } (x) \deltaplus{w}{x} + \\
  & \frac{1}{2} \frac{\partial^2 f}{\partial w^2}  \; [w, w]^{+}_\epsilon
   + \mathcal{O } (\epsilon )
  \end{flalign*}
  Taking the limit gives the expected result
  \[
  \frac{d \, f}{d x}  =  \frac{\partial f}{\partial x}   + \frac{\partial f}{\partial w} \frac{d w}{d x} \epnt
  \]
  
  \textbf{Non-Differentiable case} 
  
 This case is realized  when $ \fdiffplus{w}{x}{1/2} = K \neq 0$ and the ordinary derivative is not defined:
  In this case 
  $ w(x + \epsilon) = w(x) + (K+ \gamma ) \sqrt{\epsilon}
  $ where $\llim{\epsilon}{0}{\gamma}=0$.
  We substitute partially in the Taylor expansion:
   \begin{flalign*}
   f\left(  w(x + \epsilon), x + \epsilon \right)  - f \left( w, x \right) &  =   \frac{\partial f}{\partial x}   \; \epsilon +
   \frac{\partial f}{\partial w} (K+ \gamma ) \sqrt{\epsilon} \ + \\
  & \frac{1}{2}   \frac{\partial^2 f}{\partial x^2}   \; \epsilon^2 +
     \frac{\partial^2 f}{\partial x \, \partial w  } (x) \deltaplus{w}{x} \epsilon + \\
  & \frac{1}{2}  \frac{\partial^2 f}{\partial w^2}  \; \left( \deltaplus{w}{x}\right) ^2 
   + \mathcal{O } (\epsilon^2 ) 
   \end{flalign*}
    Then after rearrangement we get
 \begin{flalign*}      
     \frac{ f\left(  w(x + \epsilon), x + \epsilon \right)  - f \left( w, x \right) -\frac{\partial f}{\partial w} (K+ \gamma ) \sqrt{\epsilon}}{\epsilon}   =  \\
      \frac{\partial f}{\partial x}   +
      \frac{1}{2} \frac{\partial^2 f}{\partial x^2}   \; \epsilon  + 
       \frac{\partial^2 f}{\partial x \, \partial w  } (x) \deltaplus{w}{x}   + 
       \frac{1}{2}  \frac{\partial^2 f}{\partial w^2}  \;   [w, w]^{+}_\epsilon
      + \mathcal{O } (\epsilon  ) \epnt
    \end{flalign*}
  But the LHS is the forward regularized derivative of $f(w, x)$. Therefore, we obtain
  \begin{equation}
   \label{eq:taylord_fwd}
      \frac{\dj^{+} f}{d x} = \frac{\partial f}{\partial x}   + \frac{1}{2}  \frac{\partial^2 f}{\partial w^2}  \;   [w, w]^{+} 
  \end{equation}
  Therefore, we can rewrite the expansion using the uniform notation
  \begin{equation}
   \frac{\dj^{+} f}{d x} = \frac{\partial f}{\partial x}   + \frac{\partial f}{\partial w} \; \frac{\dj w}{d x}  + \frac{1}{2}  \frac{\partial^2 f}{\partial w^2}  \;   [w, w]^{+} 
  \end{equation}


\item[Backward case]

 By second order  Taylor's expansion we get
 \begin{flalign*}
 f( w(x - \epsilon), x - \epsilon) & = f(w, x) - \frac{\partial f}{\partial x}   \; \epsilon    -
 \frac{\partial f}{\partial w} \deltamin{w}{x} + \\
 & \frac{1}{2} \frac{\partial^2 f}{\partial x^2}   \; \epsilon^2 + 
 \frac{\partial^2 f}{\partial x \, \partial w  } (x) \deltamin{w}{x} \epsilon + \\
& \frac{1}{2} \frac{\partial^2 f}{\partial w^2}  \; \left( \deltaplus{w}{x}\right) ^2 
 - \mathcal{O } (\epsilon^2 )
 \end{flalign*}
 
 \textbf{Differentiable case} 
 
 When $ \fdiffmin{w}{x}{1/2} =0$ and the ordinary derivative is defined:
 \begin{flalign*}
 \diffmin{f}{w, x }{\epsilon}  = \frac{\partial f}{\partial x} + \frac{\partial f}{\partial w} \diffmin{w}{ x }{\epsilon} - \\
 \frac{1}{2} \frac{\partial^2 f}{\partial x^2} \;\epsilon -
 \frac{\partial^2 f}{\partial x \, \partial w  } (x) \deltamin{w}{x} -
 \frac{1}{2} \frac{\partial^2 f}{\partial w^2}  \; [w, w]^{-}_\epsilon
 + \mathcal{O } (\epsilon ) \epnt
 \end{flalign*}
 Taking the limit gives the expected result
 \[
 \frac{d \, f}{d x}  =  \frac{\partial f}{\partial x}   + \frac{\partial f}{\partial w} \frac{d w}{d x} \epnt
 \]
 
  \textbf{Non-Differentiable case} 
  
  when $ \fdiffmin{w}{x}{1/2} = K \neq 0$ and the ordinary derivative is not defined:
  In this case 
  $ w(x - \epsilon) = w(x) - (K \gamma ) \sqrt{\epsilon}
  $ where $\llim{\epsilon}{0}{\gamma}=0$.
  We substitute partially in the Taylor expansion:
  \begin{flalign*}
   f \left( w, x \right) - f\left(  w(x - \epsilon), x - \epsilon \right)   &  =   \frac{\partial f}{\partial x}   \; \epsilon 
  - \frac{\partial f}{\partial w} ( K+ \gamma ) \sqrt{\epsilon}  \\
  & - \frac{1}{2}   \frac{\partial^2 f}{\partial x^2}   \; \epsilon^2  
  \frac{\partial^2 f}{\partial x \, \partial w  } (x) \deltamin{w}{x} \epsilon - \\
 & - \frac{1}{2}  \frac{\partial^2 f}{\partial w^2}  \; \left( \deltaplus{w}{x}\right) ^2 
  + \mathcal{O } (\epsilon^2 ) \epnt
  \end{flalign*}
  
  Then after rearrangement we get
  \begin{flalign*}      
  \frac{f \left( w, x \right) - f\left(  w(x - \epsilon), x - \epsilon \right) +\frac{\partial f}{\partial w} (K+ \gamma ) \sqrt{\epsilon}}{\epsilon}    =   \\
  \frac{\partial f}{\partial x}   -
  \frac{1}{2} \frac{\partial^2 f}{\partial x^2}   \; \epsilon  - 
  \frac{\partial^2 f}{\partial x \, \partial w  } (x) \deltaplus{w}{x}   - 
   \frac{1}{2}  \frac{\partial^2 f}{\partial w^2}  \;   [w, w]^{+}_\epsilon
  + \mathcal{O } (\epsilon   ) \epnt
  \end{flalign*}
  But the LHS is the backwards regularized derivative of $f(w, x)$. Therefore, we obtain
  \begin{equation}
  \label{eq:taylord_bck}
  \frac{\dj^{-} f}{d x} = \frac{\partial f}{\partial x}   - \frac{1}{2}  \frac{\partial^2 f}{\partial w^2}  \;   [w, w]^{-} 
   \end{equation}
  Therefore, we can rewrite the expansion using the uniform notation
  \begin{equation}
  \frac{\dj^{-} f}{d x} = \frac{\partial f}{\partial x}   + \frac{\partial f}{\partial w} \; \frac{\dj w}{d x}  - \frac{1}{2}  \frac{\partial^2 f}{\partial w^2}  \;   [w, w]^{-} 
  \end{equation}
\end{description}

\end{proof}
From the  proof of Th. \ref{th:itotaulor} and the results of Cresson and Greff \cite{Cresson2011} it follows that:

\begin{proposition}
\label{lemma:diffitorule1}
For a composite function $ f(t, w ) \in \mathbb{C}^2$ where $w(t) \in \holder{1/2}$ the following approximation holds
\begin{align}
\deltaplus{f}{x,w}
& = \epsilon \; \frac{\partial f}{\partial x } + \sqrt{\epsilon}\; \frac{\partial f}{\partial w} \, \fdiffplus{w}{x}{1/2} + \frac{\epsilon}{2} \sum\limits^{}_{x_i,x_j} \frac{\partial ^2 f}{\partial x_j \, \partial x_j} [ x_i, \ x_j]^{+}_{\epsilon}  + 
\mathcal{O} \left( \sqrt{\epsilon} \right) \\
\deltamin{f}{x,w}
& =  \epsilon \; \frac{\partial f}{\partial x} - \sqrt{\epsilon}\; \frac{\partial f}{\partial w} \, \fdiffmin{w}{x}{1/2} - \frac{\epsilon }{2} \sum\limits^{}_{x_i,x_j} \frac{\partial ^2 f}{\partial x_j \, \partial x_j} [ x_i, \ x_j]^{-}_{\epsilon} + 
\mathcal{O} \left( \sqrt{\epsilon} \right) \ecma
\end{align}
where $\epsilon = dx $.
\end{proposition}

The last result can be used for numerical approximation of fractional singular problems.

%

\begin{example}
	Let's calculate the expansion of a quadratic function $f(v)=\dfrac{1}{2} v^2(x)$ composing with a \holder{1/2} function $v(x)$.
	\[
	\frac{\dj^{+} }{dx} f(x,v) =
	v (x)\,  \frac{ \dj^{+} v}{d\,x}  + \dfrac{1}{2} [ v,v]^{+} (x)
	\]

	For example, if $v(x)=\sqrt{x}$ we have
	for $x \neq 0$ 
	\[
	\frac{\dj^{+} }{dx} f(x,v) =\sqrt{x} \; \dfrac{1}{2 \sqrt{x}} +  \dfrac{1}{2} [ \sqrt{x},\sqrt{x}]^{+} = \dfrac{1}{2 } +0 = \dfrac{1}{2 } \ecma
	\]
	which is consistent with the  derivative of $ f(x) =\dfrac{1}{2} \, x $.
	For $x=0$ we have 
	\[
	\frac{\dj^{+} }{dx} f(x,v) = \sqrt{x} \; \cdot 0 + \dfrac{1}{2} = \dfrac{1}{2} \ecma
	\]
	which is again consistent with usual analytical treatment. 
\end{example}

\section{Discussion}
\label{sec:disc}

Classical physic variables, such as velocity or acceleration, are considered to be differentiable functions of position. On the other hand, quantum mechanical paths \cite{Abbott1981}, \cite{Amir-Azizi1987}, \cite{Sornette1990} and   Brownian motion trajectories were found to be typically non-differentiable.
The relaxation of the differentiability assumption opens new avenues in describing physical phenomena, for example, using the scale relativity theory.  

The aim of the present work is partly to explore the mathematical foundations of the theory of scale relativity championed by Nottale \cite{Nottale1989}, \cite{Nottale1998}, \cite{Nottale2011}.
The main tenet of the theory is that there is no preferred scale of description for physical reality.  
Therefore, a physical phenomenon should be described simultaneously at many irreducible scales.
This assumption naturally leads to fractal and non-differentiable representation of the studied phenomena.

Specifically, the derivative regularization procedure demonstrated here and the resulting Taylor -- It\^o expansion give support to some of the derivation procedures employed in scale relativity.

\section*{Acknowledgments}
The work has been supported in part by a grant from Research Fund - Flanders (FWO), contract number 0880.212.840.
The author would like to acknowledge Dr. Dave L. Renfro for helpful feedback on the manuscript.

\bibliographystyle{plain}  
\bibliography{qvar1}

\end{document}